\documentclass[11pt,twoside]{amsart}
\usepackage{amssymb}
\usepackage{graphicx,color}

\newtheorem{thm}{Theorem}[section]
\newtheorem{prop}[thm]{Proposition}
\newtheorem{lem}[thm]{Lemma}
\newtheorem{cor}[thm]{Corollary}
\newtheorem{defn}[thm]{Definition}

\newtheorem{ex}[thm]{Example}
\newtheorem{rem}[thm]{Remark}

\newcommand{\skipit}[1]{{}}
\newcommand{\prfend}{\hbox to7pt{\hfil}
\par\vskip-\baselineskip\hbox to\hsize
{\hfil\vbox {\hrule width6pt height6pt}}\vskip\baselineskip}

\newcommand{\Z}{\mathbb{Z}}

\newcommand{\TC}{\mathbb{T}}
\newcommand {\PP}{\mathbb{P}}

\newcommand{\cL}{\mathcal{L}}
\newcommand{\cM}{\mathcal{M}}
\newcommand{\cO}{\mathcal{O}}
\newcommand{\cE}{\mathcal{E}}

\DeclareMathOperator{\Image}{Im}

\DeclareMathOperator{\Proj}{Proj}

\newcommand{\myarrow}[2]{\hbox to #1pt{\hfil$\to$\hfil}{\hskip-#1pt{\raise
10pt\hbox to#1pt{\hfil$\scriptscriptstyle #2$\hfil}}}}
\def\rig#1{\smash{ \mathop{\longrightarrow}
    \limits^{#1}}}

\begin{document}

\title{Laplace  Equations and the Weak Lefschetz Property}
\author[Emilia Mezzetti]{Emilia Mezzetti}
\address{Dipartimento di Matematica e Informatica, Universit\`a di
Trieste, Via Valerio 12/1, 34127 Trieste, Italy}
\email{mezzette@units.it}

\author[Rosa M. Mir\'o-Roig]{Rosa M. Mir\'o-Roig}
\address{Facultat de Matem\`atiques, Department d'Algebra i
Geometria, Gran Via des les Corts Catalanes 585, 08007 Barcelona,
Spain}
\email{miro@ub.edu}

\author[Giorgio Ottaviani]{Giorgio Ottaviani}
\address{Dipartimento di Matematica, Universit\`a di Firenze,
Viale Morgagni 67/A  I-50134 Firenze, Italy }
\email{ottavian@math.unifi.it}

\begin{abstract} We prove that $r$ independent homogeneous polynomials of the same degree $d$
become dependent when restricted to any 
hyperplane if and only if their inverse system parameterizes a variety
whose $(d-1)$-osculating spaces have dimension smaller than expected. This gives an equivalence
between an algebraic notion (called Weak Lefschetz Property)
and a differential geometric notion, concerning varieties which satisfy certain Laplace equations.   In the toric case,
some relevant examples are classified and as byproduct  we provide counterexamples to Ilardi's conjecture.
\end{abstract}

\thanks{Acknowledgments:  The first author was supported by MIUR funds, PRIN project \lq\lq Geometria delle variet\`a algebriche e dei loro spazi di
moduli'' (cofin 2008). The second  author was partially   supported
by MTM2010-15256. The third author was supported by MIUR funds, PRIN project \lq\lq Propriet\`a geometriche
delle variet\`a reali e complesse'' (cofin 2009). The first and third authors are members of italian GNSAGA.
\\ {\it 2010 Mathematic Subject Classification.} 13E10,  14M25, 14N05, 14N15, 53A20.}

\maketitle

\tableofcontents

\markboth{E. Mezzetti, R. M. Mir\'o-Roig, G. Ottaviani}{Laplace equations and the Weak Lefschetz Property}

\today

\large

\section{Introduction}
The goal of this note is to establish a close relationship between  two a priori
 unrelated problems: the existence of homogeneous artinian
ideals $I\subset k[x_0, \cdots , x_n]$  which fail the Weak
Lefschetz Property; and the existence of (smooth) projective
varieties $X\subset \PP^N$ satisfying at least one Laplace equation
of order 
$s\geq 2$. These are two longstanding problems which as we will
see lie at the crossroads between Commutative Algebra, Algebraic
Geometry, Differential Geometry and Combinatorics.

A
$n$-dimensional projective variety $X\subset \PP^N$ is said to
satisfy $\delta $ independent Laplace equations of order $s$ if its
$s$-osculating space at a general point $p\in X$ has dimension
${n+s\choose s}-1-\delta $. A 
homogeneous artinian ideal $I\subset
k[x_0,x_1, \cdots ,x_n]$ is said to have the {\em Weak Lefschetz
Property (WLP)}
if there is a linear form $L \in k[x_0,x_1, \cdots ,x_n] $ such that, for all
integers $j$, the multiplication map
\[
\times L: (k[x_0,x_1, \cdots ,x_n]/I)_{j} \to (k[x_0,x_1, \cdots ,x_n]/I)_{j+1}
\]
has maximal rank, i.e.\ it is injective or surjective.
One would naively expect this property  to hold, and so it is interesting to find classes of artinian ideals failing WLP, and to understand what is from a geometric point of view that prevents this property from holding.

The starting point of this paper has been Example 3.1 in \cite{BK}
and the classical articles of Togliatti \cite{T1} and \cite{T2}. In \cite{BK}, Brenner and Kaid show
that, over an algebraically closed field of characteristic zero,
any  ideal of the form $(x^3,y^3,z^3,f(x,y,z))$, with $\deg f =
3$, fails to have the WLP if and only if $f \in
(x^3,y^3,z^3,xyz)$.  Moreover, they prove that the latter ideal is
the {\em only} such monomial ideal that fails to have the WLP. A
famous
result of Togliatti (see 
\cite{T2};  or  \cite{FI}) proves that there is only one
  non-trivial (in a sense to be precise in Section \ref{Monomial case})
    example of    surface
  $X\subset \PP^5$
  obtained by projecting the Veronese surface $V(2,3)\subset \PP^{9}$
  and satisfying a single Laplace equation of order 2; $X$ is
  projectively equivalent to the image of $\PP^2$ via the
  linear system $ \langle x^2y,xy^2,x^2z,xz^2,y^2z,yz^2 \rangle \subset | \cO_{\PP^2}(3)|$.
  Note that the linear system of cubics given by Brenner and Kaid's
  example $\langle x^3,y^3,z^3,xyz \rangle $ is apolar  to the linear system
  of cubics given in Togliatti's example. A  careful
   analysis of this example suggested us that there is
   relationship between artinian ideals $I\subset k[x_0,
   \cdots ,x_n]$ generated by $r$ homogeneous forms of degree
   $d$ that fail Weak Lefschetz Property and  projections of the
   Veronese variety $V(n,d)\subset \PP^{{n+d \choose d}-1}$ in
   $X\subset \PP^{{n+d\choose d}-r-1}$ satisfying at least a Laplace
   equation of order $d-1$. Our goal will be to exhibit such
   relationship with the hope to  shed  more light on these
   fascinating and perhaps intractable problems of classifying
   the artinian ideals which fails the Weak Lefschetz property
   and of classifying $n$-dimensional projective varieties satisfying
at least one Laplace equation of order $s$.
Our main theorem \ref{teathm} says that an ideal $I$ generated by homogeneous forms of degree $d$, satisfying some reasonable assumptions,  fails the WLP in degree $d-1$ if and only if its apolar ideal $I^{-1}$ parameterizes a variety which satisfies a Laplace equation of degree $d-1$.

{\bf Notation.}
$V(n,d)$ will denote the image of the projective space $\PP^n$ in the $d$-tuple
Veronese embedding $\PP^n\longrightarrow \PP^{{n+d\choose d}-1}$.
$(F_1,\ldots,F_r)$ denotes the ideal generated by $F_1,\ldots,F_r$,
while $\langle F_1,\ldots,F_r \rangle$
denotes the vector space they generate.

\vskip 2mm
Next we outline the structure of this note.  In
section~\ref{defs and prelim results} we  fix the notation
and we collect the basic results on Laplace equations and
Weak Lefschetz Property needed in the sequel.
Section~\ref{Tea-THM} is the heart of the paper. In this section, we state and
prove our main result (see Theorem \ref{teathm}).
In section~\ref{Monomial case}, we restrict our attention to
the monomial case and we give a complete classification in the case of smooth and quasi-smooth
cubic linear systems on $\PP^n$ for $n\le 3$. In section \ref{bounds} we concentrate
in the case $n=2$ and specifically on ideals with $4$ generators.
We end the paper in
section~\ref{final comments} with some natural problems coming up
from our work and a family of counterexamples to Ilardi's conjecture which work for
any $n\ge 3$.

\vskip 2mm \noindent {\bf Acknowledgement.} This work began in
Winter 2009 during our visit to the MSRI in Berkeley. We thank the organizers
of the Algebraic Geometry program for their kind
hospitality, and Rita Pardini with whom we shared tea and discussed Tea Theorem.

\section{Definitions and preliminary results} \label{defs and prelim results}

In this section we recall some standard
terminology and notation
from commutative algebra and algebraic geometry,
as well as some results needed in the sequel.

Set $R = k[x_0,x_1,\dots,x_n]$ where $k$ is an
algebraically closed field of characteristic zero
and let  $\mathfrak{m} = (x_0,x_1, \cdots ,x_n)$
be its maximal homogeneous ideal. We
consider  a homogeneous ideal $I$ of $R$.
The Hilbert
function $h_{R/I}$ of $R/I$ is defined
by $h_{R/I}(t):=\dim _k(R/I)_{t}$.
Note that the Hilbert function of an artinian $k$-algebra $R/I$ has
finite support and is captured in its {\em h-vector} $\underline{h} =
(h_0,h_1,\dots,h_e)$ where $h_0=1$, $h_{i}=h_{R/I}(i)>0$ and $e$ is the
last index with this property.

In the case of three variables, we will often use $x,y,z$ instead of $x_0,x_1,x_2$.

\vskip 4mm
\noindent
{\bf A. The Weak Lefschetz Property}

\begin{defn}\label{def of wlp}\rm Let $I\subset R$ be a
homogeneous artinian ideal. We will say that  the
standard graded artinian algebra $R/I$
 has the {\em Weak Lefschetz Property (WLP)}
if there is a linear form $L \in (R/I)_1$ such that, for all
integers $j$, the multiplication map
\[
\times L: (R/I)_{j} \to (R/I)_{j+1}
\]
has maximal rank, i.e.\ it is injective or surjective.
(We will often abuse notation and say that the ideal $I$ has the
WLP.)   In this
case, the linear form $L$ is called a {\em Lefschetz element} of
$R/I$.  If for the general form $L \in (R/I)_1$
and for an integer number $j$
the map $\times L$ has not maximal rank we will say that 
the ideal $I$ fails the WLP in degree $j$.\end{defn}

The Lefschetz elements of $R/I$ form a Zariski
open, possibly empty, subset of $(R/I)_1$.   Part of the great
interest in the WLP stems from the fact that its presence puts
severe constraints on the possible Hilbert functions, which can
appear in various disguises (see, e.g.,
\cite{St-faces}). Though many algebras 
 are expected to have the
WLP, establishing this property is often rather difficult. For
example, it was shown by R. Stanley \cite{st} and J. Watanabe
\cite{w} that a monomial artinian complete intersection ideal
$I\subset R$ has the WLP. By semicontinuity, it follows that
a {\em general}   artinian complete intersection ideal $I\subset
R$ has the WLP but it is open whether {\em every} artinian complete
intersection of height $\ge 4$
over a field of characteristic zero has the WLP. It is worthwhile to point out that in positive characteristic, there are examples of artinian complete
intersection ideals $I\subset k[x,y,z]$ failing the WLP (see, e.g., Remark 7.10 in \cite{MMN2}).

\begin{ex}\rm (1) The ideal $I=(x^3,y^3,z^3,xyz)\subset k[x,y,z]$
fails to have the WLP because for any linear form $L=ax+by+cz$
the multiplication map $$\times L:(k[x,y,z]/I)_2 \rightarrow (k[x,y,z]/I)_{3}$$
is neither injective nor surjective.
Indeed, since it is a map between two $k$-vector spaces of dimension 6,
to show the latter assertion it is enough to exhibit a non-trivial
element in its kernel. Take $f=a^2x^2+b^2y^2+c^2z^2-abxy-acxz-bcyz$.
$f$ is not in $I$
and we easily check that $L \cdot f$ is in $I$.

(2) The ideal $I=(x^3,y^3,z^3,x^2y )\subset k[x,y,z]$ has the WLP.
Since the $h$-vector of $R/I$ is (1,3,6,6,4,1), we only need to
check that  the map $\times L : (R/I)_{i} \rightarrow (R/I)_{i+1}$
induced by $L=x+y+z$ is surjective for $i=2,3,4$. This is equivalent
to check that $(R/(I,L))_{i}=0$ for $i=3,4,5$. Obviously, it is enough
to check the case $i=3$. We have

\[
\begin{array}{lll}
(R/(I,L))_3 & \cong  (k[x,y,z]/(x^3,y^3,z^3,x^2y,x+y+z))_3 \\ & \cong
(k[x,y]/(x^3,y^3,x^3+3x^2y+3xy^2+y^3, x^2y))_3 \\ &
\cong  k[x,y]/(x^3,y^3,x^2y,xy^2))_3=0 \end{array}\]
which proves what we want.
\end{ex}

In this note we are mainly interested in artinian ideals
$I$ generated by homogeneous forms of fixed degree $d$.
In this case we have the following easy but useful lemma.

\begin{lem} \label{keylemma} Let $I\subset R=k[x_0,x_1,
\cdots ,x_n]$ be an artinian
ideal generated by $r\le {n+d-1\choose d}$ homogeneous
forms $F_1, \cdots ,F_r$ of degree $d$.
Let $L$ be a general linear form, let $\bar R = R/(L)$
and let
$\bar I$ (resp. $\bar F_{i}$) be the image of $I$
(resp. $F_{i}$) in $\bar R$.
Consider the homomorphism $\phi_{d-1} : (R/I)_{d-1}
\rightarrow (R/I)_{d}$ defined by
multiplication by $L$.
Then $\phi _{d-1}$ has not maximal rank  if and
only if  $\bar F_1, \cdots ,\bar F_r$ are $k$-linearly dependent.
\end{lem}

\begin{proof} First note that $(R/I)_{d-1}\cong R_{d-1}$,
$\dim R_{d-1}={n+d-1\choose d-1}$, $\dim (R/L)_d={n+d-1
\choose n-1}$ and $\dim (R/I)_d={n+d\choose d}-r$. Consider the exact sequence
\[
0 \rightarrow \frac{[I:L]}{I} \rightarrow R/I \stackrel{
\times L}{\longrightarrow} (R/I) (1) \rightarrow (R/(I,L))(1) \rightarrow 0
\]
where $\times L$ in degree $d-1$ is just $\phi_{d-1}$.
This shows that the cokernel of $\phi_{d-1}$ is just $(R/(I,L))_{d}$.

Since $r\le {n+d-1\choose d}$, we  have $\dim (R/I)_{d-1}\le
\dim (R/I)_d$. Hence, $\phi_{d-1}$ is not of maximal rank if and
only if  $\phi_{d-1}$ is not injective, if and only if $rk(
\phi _{d-1})< {n+d-1\choose d-1}$, if and only if $
\dim (R/(I,L))_d=\dim(\bar R)_d-\dim \bar I_d={n+d-1
\choose n-1}-\dim  \langle\bar F_1,\cdots ,\bar F_r \rangle_d \gneqq
\dim (R/I)_d-{n+d-1\choose d-1}= {n+d\choose d}-{n+d-1
\choose d-1}-r={n+d-1\choose n-1}-r$. Therefore, $\phi _{d-1 }$
is not injective if and only if $\dim \langle\bar F_1,\cdots ,\bar F_r\rangle
\lneqq r$, if and only if $\bar F_1, \cdots ,\bar F_r$ are
$k$-linearly dependent. \end{proof}

As an easy consequence we have the following useful corollary.

\begin{cor} Let $F_1, \cdots ,F_r\in R=k[x_0,x_1, \cdots ,x_n]$ be
a set of $\mathfrak{m}$-primary homogeneous forms of degree $d$.
Let $L$ be a general linear form, let $\bar R = R/(L)$ and let
$\bar F_{i}$ be the image of $F_{i}$ in $\bar R$.  If
$r\le {n-1+d\choose d}$ and $\bar F_1, \cdots ,\bar F_r$
are $k$-linearly dependent, then the ideal $I=(F_1, \cdots ,
F_r)$ fails WLP and, moreover, the same is true for any
enlarged ideal $J=(F_1, \cdots ,F_r, F_{r+1},
\cdots ,F_{t})\varsubsetneq R_d$ with $r\le t\le {n-1+d\choose d}$.
\end{cor}

Closing this subsection, we reformulate the Weak Lefschetz Property by using the theory
of vector bundles on the projective space and we refer to \cite{BK} for more information.

To any subspace $\langle F_1,\ldots , F_r\rangle $ generated by $r$ $\mathfrak{m}$-primary homogeneous forms of degree $d$,
it is associated a kernel vector bundle $K$ as in the following exact sequence on $\PP^n$
$$0\rig{}K\rig{}\cO^r\xrightarrow{F_1,\ldots , F_r}\cO(d)\rig{}0$$

The fact that $K$ is locally free is equivalent to $(F_1,\ldots , F_r)$
 being $\mathfrak{m}$-primary.

It is well known that the bundle $K$ splits on any line $L\subset \PP^n$ as the sum of line bundles.
On the general line $L$ we have a splitting
$K_{|L}\simeq \oplus_{i=1}^{r-1}\cO_L(a_i)$, where $a_{i}\le 0$ for $1\le i \le r-1$ and, moreover, we may assume that $a_1\le\ldots\le a_{r-1}$.
The $(r-1)$-ple $(a_1,\ldots a_{r-1})$ is called the generic splitting type of $K$.

\begin{thm}
Let $I=(F_1,\ldots , F_r)$ be a $\mathfrak{m}$-primary ideal generated by $r$  homogeneous forms
and let $(a_1,\ldots a_{r-1})$ be the generic splitting type of the kernel bundle $K$.
The following properties are equivalent:
\begin{itemize}
\item[i)] $I$ has the WLP;
\item[ii)] $a_{r-1}<0$.
\end{itemize}
\end{thm}
\begin{proof} The forms $F_1,\ldots F_r$ restricted to a general line $L$ are dependent if and only if
the restricted map $H^0(\cO_L^r)\xrightarrow{F_1,\ldots , F_r}H^0(\cO_L(d))$ has a nonzero kernel.
The result follows because the kernel is   $\oplus H^0(\cO_L(a_i))$.
\end{proof}

In the Togliatti's example we get as kernel a rank three vector bundle on $\PP^2$ with generic splitting type
$(-2,-1,0)$.

\vskip 4mm
\noindent
{\bf B. Laplace Equations}

\vskip 2mm \noindent
In this section we adopt the point of view of differential geometry, for instance as in \cite{GH}.

Let $X\subset \PP^N$ be a quasi-projective variety of dimension $n$.
Let $x\in X$ be a smooth point. We can choose a system of affine
coordinates around $x$ and a local parametrization of $X$ of the
form $\phi(t_1,...,t_n)$ where $x=\phi(0,...,0)$ and the $N$
components of $\phi$ are formal power series.

The tangent space to $X$ at $x$ is the $k$-vector space generated by the $n$
vectors which are the partial derivatives
of $\phi$ at $x$.
Since $x$ is a smooth point of $X$  these $n$ vectors are $k$-linearly independent.
Note that this is not the tangent space in the Zariski
sense, but in differential-geometric language.

Similarly one defines the $s$th osculating (vector) space
$T_x^{(s)}X$ to be the span of all
partial derivatives of $\phi$ of order
$\leq s$ (see for instance \cite{GH}).
The expected dimension of $T_x^{(s)}X$
is ${{n+s}\choose {s}}-1$, but
in general $\dim T_x^{(s)}X\leq {{n+s}\choose {s}}-1$;
if strict inequality
holds for all smooth points of $X$, and $\dim
T_x^{(s)}X={{n+s}\choose {s}}-1-\delta$ for general $x$,
then $X$ is said to satisfy $\delta$ Laplace equations of order $s$.
Indeed, in this case the partials of order $s$ of $\phi$
are linearly dependent, which gives $\delta$ differential equations
of order $s$ which are satisfied by the components of $\phi$.
We will also consider the projective $s$th osculating space $\TC_x^{(s)}X$, embedded in
$\PP^N$.

\begin{rem}\label{bound} \rm It is clear that if $N<{{n+s}\choose {s}}-1$ then $X$
satisfies at least one Laplace equation
of order $s$, but this case is not interesting and
will not be considered in the following.
\end{rem}

\begin{rem}\label{ruled} \rm
If $X$ is uniruled by lines, i.e. through any general point of $X$ passes a line
contained in $X$, then $X$ satisfies a Laplace equation. Indeed in this case it is possible to find
a parametrization of $X$ in which one of the parameters appears at most at degree one.
Hence the corresponding second derivative vanishes identically.
\end{rem}

If $X\subset \PP^N$ is a rational variety, then there exists a birational map
$\PP^n \dashrightarrow X$ given by $N+1$ forms $F_0, \ldots, F_N$ of degree $d$ of
$k[x_0,x_1, \cdots ,x_N]$. From Euler's formula it follows that
the projective $s$th osculating space $\TC_x^{(s)}X$, for $x$ general, is generated by the $s$-th partial derivatives
of $(F_0, \cdots ,F_N)$ at the point $x$.

Assume that $X$ is not a linear space. In the case $s=2$, $n=2$,
the dimension of $\TC_x^{(2)}X$ varies between $3$ and $5$.
Moreover
$\dim \TC_x^{(2)}X=3$ for general $x\in X$ if and only if $X$ is either
a hypersurface or a ruled developable surface, i.e. a cone or
the developable tangent of a curve. The surfaces with $\dim \TC_x^{(2)}X=4$
for general $x\in X$ are not well understood yet, in spite of the literature devoted to
this topic (see \cite {S}, where they are called \lq\lq superfici $\Phi$'', \cite{GH}, page 377, \cite{FI} \cite{I}
\cite{T1}, \cite{T2},
\cite{V}, \cite{MT}).  If $X\subset \PP^N$ with $N\geq 5$ is not a Del Pezzo surface,
i.e. $X$ is not a projection of $V(2,3)$,
besides the ruled surfaces, there are only few examples; in particular among the
known smooth examples there are the Togliatti surface introduced above (see the Introduction),
a special 
complete intersection of quadrics in $\PP^5$
desingularization of the Kummer surface (see \cite{D} and \cite{E})
and some toric surfaces (see the examples given by Perkinson in \cite{P}, where the classification is given
of toric surfaces and threefolds whose osculating spaces up to order $d-1$ have all maximal dimension
and have all dimension less than maximal for order $d$).



\section{The Main Theorem} \label{Tea-THM}

The goal of this section is to  highlight the existence of
a surprising relationship between a pure algebraic problem:
the existence of artinian ideals $I\subset R$ generated by homogeneous forms
of degree $d$ and failing the WLP; and a pure geometric problem:
the existence of  projections of the Veronese variety
$V(n,d)\subset \PP^{{n+d\choose d}-1}$ in $X\subset \PP^{N}$
satisfying at least one Laplace equation of order $d-1$.
Moreover, we will also discuss   the
geometry of some surfaces \lq\lq apolar" to those satisfying
the Laplace equation.

We start this section recalling the basic facts on {\em Macaulay-Matlis
duality}
which will allow us to relate the above mentioned problems.
Let $V$ be an $(n+1)$-dimensional $k$-vector space and
set $R=\oplus _{i\ge 0}Sym^{i}V^{*}$ and $\mathcal
R=\oplus _{i\ge 0}Sym^{i}V$. 
Let $\{x_0,x_1,\cdots ,x_n\}$, $\{y_0,y_1,\cdots,y_n\}$ be dual
bases of $V^*$ and $V$ respectively. So, we have
the identifications $R=k[x_0,x_1, \cdots ,x_n]$ and
$\mathcal R=k[y_0,y_1,\cdots ,y_n]$. There are products (see
\cite{FH}; pg. 476)
\[\begin{array}{ccc} Sym^{j}V^{*} \otimes Sym^{i}V & \longrightarrow
& Sym^{i-j}V \\
u\otimes F & \mapsto & u\cdot F
\end{array}
\] making $\mathcal R$ into a graded $R$-module. We can see this action
as partial differentiation: if $u(x_0, x_1,\cdots ,x_n)\in R$ and
$F(y_0,y_1,\cdots ,y_n)\in \mathcal R$, then $$u\cdot
F=u(\partial/\partial y_0,\partial/\partial y_1, \cdots ,
\partial/\partial y_n)F.$$ If $I\subset R$ is a homogeneous
ideal, we define  the {\em Macaulay's inverse system} $I^{-1}$ for $I$ as
   $$I^{-1}:=\{F\in \mathcal R, u\cdot F=0 \text{ for all } u\in I \}.$$
   $I^{-1}$ is an $R$-submodule of $\mathcal R$ which inherits a
grading of $\mathcal R$.
Conversely, if $M\subset \mathcal R$ is a graded $R$-submodule,
then $Ann(M):=\{u\in R, u\cdot F =0 \text{ for all } F\in M \}$
is a homogeneous ideal in $R$. In classical terminology,
if $u\cdot F=0$ and $deg(u) = deg(F)$, then $u$ and $F$ are
said to be {\em apolar} to each other. In fact, the
pairing $$R_{i}\times \mathcal R _{i} \longrightarrow k \quad
\quad  (u,f)\mapsto u\cdot F  $$ is exact; it is called the
apolarity or Macaulay-Matlis duality action of $R$ on $\mathcal R$.

For any integer $i$, we have $h_{R/I}(i)= \dim _k(R/I)_{i}=
\dim_k (I^{-1})_{i}$. The following Theorem is fundamental.

\begin{thm}\label{macaulat-matlis} We have a bijective correspondence
\[\begin{array}{ccc}  \{ \text{ Homogeneous ideals } I\subset R \} &
\rightleftharpoons & \{ \text{ Graded } R-\text{submodules of }\mathcal R \} \\
I & \rightarrow & I^{-1} \\ Ann(M) & \leftarrow & M \end{array} .\]
Moreover, $I^{-1}$ is a finitely generated $R$-module if and only if
$R/I$ is an artinian ring.
\end{thm}

\vskip 2mm
When considering only monomial ideals, we can simplify by
regarding the inverse system in the same polynomial ring
$R$,  and in any degree, $d$, the inverse system $I^{-1}_d$
is spanned by the monomials in $R_d$ not in $I_d$. Using the
language of inverse systems, we will still call the elements
obtained by the action \emph{derivatives}.

\vskip 2mm

Let $I$ be an artinian
ideal
generated by $r$
  homogeneous polynomials 
  $F_1, \cdots ,F_r\in R=k[x_0,x_1,
  \cdots ,x_n]$ of  degree $d$. Let $I^{-1}\subset
  \mathcal R$ be its Macaulay inverse system. Associated to
  $(I^{-1})_d$ there is a rational map 
$$\varphi _{(I^{-1})_d}:\PP^n \dashrightarrow \PP^{{n+d
\choose d}-r-1}.$$ Its image $\overline{\Image (
\varphi _{(I^{-1})_d})}\subset \PP^{{n+d\choose d}-r-1}$
is the projection of the $n$-dimensional Veronese variety
$V(n,d)$ from the linear system $\langle F_1,\cdots ,F_r \rangle \subset |
\cO _{\PP^n}(d)|$. Let us call it $X_{n,(I^{-1})_d}$.
Analogously,  associated to $I_d$ there is a morphism
$$\varphi _{I_d}:\PP^n \longrightarrow \PP^{r-1}.$$
Note that $\varphi _{I_d}$ is regular because $I$ is artinian. 
Its image ${\Image (\varphi _{I_d})}\subset \PP^{r-1}$
is the projection of the $n$-dimensional Veronese variety $V(n,d)$
from the linear system $\langle(I^{-1})_d \rangle\subset | \cO _{\PP^n}(d)|$.
Let us call it $X_{n,I_d}$. The varieties  $X_{n,I_d}$ and
$X_{n,(I^{-1})_d}$ are usually called  apolar.

We are now ready to state the main result of this section. We have:

\begin{thm}\label{teathm}[The Tea Theorem] Let $I\subset R$ be an artinian
ideal
generated
by $r$ homogeneous polynomials $F_1,...,F_{r}$ of degree $d$.
If 
$r\le {n+d-1\choose n-1}$, then
  the following conditions are equivalent:
\begin{itemize}
\item[(1)] The ideal $I$ fails the WLP in degree $d-1$,
\item[(2)] The  homogeneous forms $F_1,...,F_{r}$ become
$k$-linearly dependent on a general hyperplane $H$ of $\PP^n$,
\item[(3)] The $n$-dimensional   variety
$X_{n,(I^{-1})_d}$ satisfies at least one Laplace equation of order $d-1$.
\end{itemize}
\end{thm}

\begin{rem} \rm
Note that, in view of Remark \ref{bound}, the assumption
$r\le {n+d-1\choose
n-1}$ ensures that the Laplace equations obtained in (3)
are not obvious in the sense of Remark \ref{bound}. In the particular case $n=2$, this assumption
gives  $r\leq d+1$.
\end{rem}
\begin{rem}\rm Since the first guess about the statement of the Theorem
emerged during a Tea discussion in Berkeley, we always labeled the result
in our discussions
as the Tea Theorem.
\end{rem}
\begin{proof} The equivalence between (1) and (2)
follows immediately from Lemma \ref{keylemma}. Let us see that (1) is equivalent to (3). Since
$(R/I)_{d-1}=R_{d-1}$ and $\dim R_{d-1}={n+d-1\choose n}=
{n+d\choose n}- {n+d-1\choose n-1}\le {n+d\choose n}-r=\dim
(R/I)_d$, we have that the ideal $I$ fails the WLP in degree $d-1$
if and only if for a general linear form $L\in R_1$ the
multiplication map
\[
\times L: (R/I)_{d-1} \to (R/I)_{d}
\]
is not injective. Via the Macaulay-Matlis duality, the latter is
equivalent to say that the rank of the dual map $(I^{-1})_d
\longrightarrow (I^{-1})_{d-1}$ is $\le {d+n-1\choose n}-1$; which
is equivalent to say that the $(d-1)$-th osculating space
$\TC_x^{(d-1)}X_{n,(I^{-1})_d}$ spanned by all partial derivatives of
  order $\le d-1$ of the given parametrization of $X_{n,(I^{-1})_d}$
  has dimension $\le
{n+d-1\choose n}-2$, i.e. $X_{n,(I^{-1})_d}$ satisfies a Laplace
equation of order $d-1$.
\end{proof}

\begin{rem}\label{Tog} \rm Note that for  $n=2$, $d=3$ and
$I=(x_0^3,x_1^3,x_2^3,x_0x_1x_2)\subset k[x_0,x_1,x_2]$,
  we recover Togliatti's example (see \cite{T1}, \cite{T2} and \cite{FI}).
\end{rem}

\begin{defn} \rm With notation as above,
we will say that $I^{-1}$ (or $I$) defines a \emph{Togliatti system}
if it satisfies the three equivalent conditions in  Theorem \ref{teathm}.
\end{defn}

\begin{ex} \rm (see \cite{V}) 
Let $d=2k+1$ be an odd number and $n=2$. Let $l_1, \ldots, l_d$ be
general linear forms in $3$ variables. Then the ideal
$(l_1^d,\ldots,l_d^d, l_1l_2\cdots l_d)$ is generated by $d+1$
polynomials of degree $d$ and it fails the WLP in degree $d-1$
because by \cite{V}, Th\'{e}or\`{e}me 3.1, $l_1^d,\ldots,l_d^d,
l_1l_2\cdots l_d$ become dependent on a general line $L\subset
\PP^2$. For $d=3$ we recover Togliatti example once more, for
$d>3$ we get non-toric examples. It is interesting to observe that
a similar construction in even degree produces ideals 
which do satisfy
the WLP.
\end{ex}

\begin{ex}\label{trivB}  \rm
Let $n\geq 3$ and $d\geq 3$. Let $I=(LF_1,...,LF_t, G_1,\ldots,G_n)$ where $L$ is a linear form,
$F_1,\ldots,F_t$ are general forms of degree $d-1$ and $G_1,\ldots,G_n$
general forms of degree $d$. If $ {{n+d-2}\choose {n-1}}+1\leq t\leq {{n+d-1}\choose{n-1}}-n$,
then $I$ is artinian and fails WLP in degree $d-1$.  Indeed the number of conditions imposed 
to the forms of degree $d-1$ 
to contain a linear form  is equal to ${{n+d-2}\choose {n-1}}$. With the assumptions made on $t$,
the number of generators  $r=t+n$ is in the range of Theorem \ref{teathm}.
\end{ex}

We will end this section  studying the geometry of some rational surfaces
satisfying at least one Laplace equation of order $2$ and the geometry
of their apolar surfaces.

\begin{ex} \rm In the case of the Togliatti surface the  morphism
  $\varphi _{I_3}: \PP^2 \longrightarrow \PP^{3}$ with
$I_3=(x_0^3,x_1^3,x_2^3,x_0x_1x_2)$ is  not birational. In fact, it
is a triple cover of the
  cubic surface of equation $xyz=t^3$, which is singular at the three
  fundamental points of the plane $t=0$.

  Similarly in the case $n=2$, $d=4$ and $I_4=(x_0^4,x_1^4,x_2^4,
x_0^2x_1^2, x_0x_1x_2^2)$, 
the surface $X_{2,(I^{-1})_4}\subset \PP^9$
has second osculating space of dimension $8$ at a general point.
Also the morphism
$\varphi _{I_4}:\PP^2 \longrightarrow \PP^4$
is not birational; it is a degree $4$ cover of a singular Del Pezzo  quartic,
complete intersection of two quadrics in $\PP^4$.

Similar considerations can be made in the following example, where $n=2$, $d=5$ and $I_5=(x_0^5,x_1^5,x_2^5,
x_0^3x_1^2, x_0^2x_1^3, x_1^2x_2^3)$, but in this case
we get a birational map $\varphi _{I_5}:\PP^2 \longrightarrow \PP^5$.
\end{ex}


\section{The Toric Case} \label{Monomial case}

In this section, we will restrict our attention
to the monomial case. First of all, we want to point
out that for monomial ideals (i.e. the ideals invariants for the natural toric action of $(k^*)^n)$ on $k[x_0,\ldots ,x_n]$)
to test the WLP there
is no need to consider a general linear form. In fact, we have

\begin{prop}
   \label{lem-L-element}
Let $I \subset R:=k[x_0,x_1,\cdots , x_n]$ be an artinian monomial ideal.
Then $R/I$ has the WLP if and only if  $x_0+x_1 + \cdots + x_n$
is a Lefschetz element for $R/I$.
\end{prop}

\begin{proof} See \cite{MMN2};  Proposition 2.2.
\end{proof}

Fix $\PP^n=\Proj(k[x_0,x_1,\cdots ,x_n])$. Denote by $\cL _{n,d}:=|
\cO_{\PP^n}(d)|$
the complete linear system of hypersurfaces of degree $d$ in $\PP^n$
and set $n_d:=\dim (\cL_{n,d})={n+d\choose n}-1$ its projective dimension.
As usual denote by $V(n,d)\subset \PP^{n_d}$ the Veronese variety.

\begin{defn}\rm A linear subspace $\cL \subset \cL _{n,d}$ is called
a {\em monomial linear subspace}
if it can be generated by monomials.
\end{defn}

\vskip 4mm
\noindent {\bf The example of the truncated simplex:}
Consider the linear system of cubics
$$\cL =|\{ x_{i}^2x_{j} \}_{0\le i\ne j\le n}| \subset \cL_{n,3}.$$
Note that $\dim \cL =n(n+1)-1$.  Let $$\varphi _{\cL}:\PP^n  \dashrightarrow
\PP^{n(n+1)-1}$$ be the rational map associated to $\cL$.
Its image $X:=\overline{\Image (\varphi _{\cL})}
\subset
\PP^{n(n+1)-1}$ is (projectively equivalent to)
the projection of the Veronese variety $V(n,3)$ from
the linear subspace $$\cL':= |\langle x_0^3, x_1^3, ...,x_n^3,
\{x_ix_jx_k\}_{0\le i<j<k\le n} \rangle | $$ of $\PP^{{n+3\choose 3}-1}$.
Let us check that $X$ satisfies a Laplace equation of order 2 and
that it is smooth.

Since $\cL $ and $\cL '$ are apolar, we can apply Theorem \ref{teathm}
and we get  that $X$ satisfies a Laplace equation of order 2
if and only if the ideal $I=(x_0^3, x_1^3, ...,x_n^3,
\{x_ix_jx_k\}_{0\le i<j<k\le n})\subset R=k[x_0,x_1,\cdots ,x_n]$
fails the WLP in degree $2$, i.e. for a general linear form $L\in R_1$
the map $\times L:(R/I)_2 \longrightarrow (R/I)_3$ has not maximal rank.
By Lemma \ref{keylemma}, it is enough to see that the restriction
of the cubics $x_0^3, x_1^3, ...,x_n^3, \{x_ix_jx_k\}_{0\le i<j<k
\le n}$ to a general hyperplane become $k$-linearly dependent and, by
Proposition \ref{lem-L-element}, it is enough to check that they
become $k$-linearly dependent when we restrict to the hyperplane
$x_0+x_1+ \cdots +x_n=0$, which follows after a straightforward computation.
An alternative argument, due to the Proposition 1.1 of \cite{P},
is that all  the vertices points in $\Z^{n+1}$, corresponding to the monomial basis of $\cL$, are contained in the quadric
with equation $2\left(\sum_{i=0}^nx_i^2\right)-5\left(\sum_{0\le i<j\le n}x_ix_j\right)=0$.

$X$ is a projection of the blow-up of $\PP^n$ at the $n+1$ fundamental points,
embedded via the linear system of cubics passing through the blown-up  points.
Using the language of \cite{GKZ},
 it is the projective toric variety $X_A$, associated to the set $A$ of vertices of the
 lattice polytope $P_n$  defined as follows: let $\Delta_n$ be the standard simplex
 in $\mathbb R^n$, consider $3\Delta_n$,
 then $P_n$ is obtained by removing all vertices so that the new edges have all length one:
 $P_n$ is a \lq\lq truncated simplex''.
 By the smoothness criterium  Corollary 3.2, Ch. 5, in \cite{GKZ} (see also \cite{P}), it follows that $X$ is smooth.
 For instance, in the case $n=2$ $P_2$ is the punctured hexagon of Figure 2.

\vskip 4mm

In \cite{I}, pag. 12, G. Ilardi
formulated a conjecture, stating that
the above example is the only smooth (meaning that the variety $X$ is smooth) monomial Togliatti system
of cubics of dimension $n(n+1)-1$.
We will show that the conjecture is incorrect, but we underline that it was useful to us because
it pointed in the right direction.

We start by producing a  class of examples of monomial Togliatti systems of cubics, holding for any $n\ge 3$,
we will then give the classification  of smooth and quasi-smooth monomial Togliatti systems for $n=3$ in the Theorem \ref{toricn3}.
As a consequence, the conjecture in \cite{I} at page 12 cannot hold, in the sense that that the list
in \cite{I} is too short and we have to enlarge it. Correspondingly, in the remark \ref{finalrem}, we
propose a larger list for any $n$, which reduces to the list of  the Theorem \ref{toricn3} for $n=3$.

\noindent {\bf A second example:} Consider the linear system of cubics
$\cM =| \{ x_{i}^2x_{j} \}_{0\le i\ne j\le n, \{i,j\}\neq \{0,1\}}\cup\{x_0x_1x_i\}_{2\le i\le n}| \subset \cL_{n,3}$.
Note that $\dim \cM =n^2+2n-
4$.
Let $$\varphi _{\cM}:\PP^n  \dashrightarrow
\PP^{n^2+2n-4}$$ be the rational map associated to $\cM$.
Its image $X:=\overline{\Image (\varphi _{\cM})}\subset
\PP^{n^2+2n-4}$ is (projectively equivalent to)
the projection of the Veronese variety $V(n,3)$ from
the linear subspace $$\cM':=|\langle x_0^3, x_1^3, ...,x_n^3,x_0^2x_1,x_0x_1^2,
\{x_ix_jx_k\}_{0\le i<j<k\le n, (i,j)\neq (0,1)} \rangle |$$ of $ \cM
_{n,3}=\PP^{{n+3\choose 3}-1}$.
Arguing as in the previous example we can check  that $X$ satisfies a Laplace equation of order 2 and
that it is smooth.
The quadric containing all the vertices points
 in $\Z^{n+1}$ has equation
$2\left(\sum_{i=0}^nx_i^2\right)-5\left(\sum_{0\le i<j\le n}x_ix_j\right)+9x_0x_1=0$.

Notice that  $n^2+2n-4=n^2+n-1$ if and only if $n=3$. Hence for $n=3$ we have got a counterexample to Ilardi's conjecture.
Nevertheless $X$ cannot be further projected without acquiring singularities; hence,
 for $n>3$ this example does not give  a counterexample to Ilardi's conjecture.
See section 6 for counterexamples to Ilardi's conjecture for any $n\ge 3$.

 $X$ is a projection of the blow-up of $\PP^n$ at  $n-1$ fundamental points plus
the line through the remaining two fundamental points,
embedded via a linear system of cubics.
Also in this case, as in the previous one,
$X$ is a projective toric variety of the form $X_A$. Now there is a
 lattice polytope $P$  obtained from $3\Delta_n$,
 removing $n-1$ vertices and the  opposite edge. $A$ is the set of the vertices of $P$
 together with the $n-1$ central points of the $2$-faces adjacent to the removed edge.
By the above smoothness criterium, $X$ cannot be further projected without acquiring singularities.

\vskip 4mm

\subsection{Geometric point of view and trivial linear systems}

With notation as in Section \ref{Tea-THM}, we consider now a monomial artinian
ideal $I$, generated by a subspace $I_d\subset Sym^dV^*$ (where $V\simeq{\mathbb C}^{n+1}$).  Since we are in the monomial case, we will also assume $I^{-1}_d\subset  Sym^dV^*$.
\begin{rem} \rm \label{sing} Note that the assumption that $I$ is artinian is equivalent to  $I^{-1}\cap V(n,d)=\emptyset$. Indeed, if
$I$ is not artinian, then there exists a point $z\in \PP^n$ which is a common zero of all polynomials in $I$.
Then its Veronese image $v_d(z)$ belongs to $V(n,d)\cap I^{-1}$. Here $v_d(z)$ must be interpreted as $\sum z_\alpha
\partial_\alpha$ where $\alpha$ denotes a multiindex of degree $d$. Conversely, if $v_d(z) \in I^{-1}$, then $(\sum z_\alpha \partial_\alpha)(F)=0$ for all $F\in I_d$, therefore, being $I$ generated by $I_d$, $z$ is a common zero of the polynomials of $I$.
\end{rem}

Let $X$ be the closure of the image of $\varphi_{I_d^{-1}}$, it can be seen geometrically as the projection of $V(n,d)$ from $I_d$.
The exceptional locus of this projection is $I\cap V(n,d)$ and corresponds via $v_d$ to the base locus  of the linear system $\langle I_d^{-1} \rangle$.  $X$ can also be interpreted as (a projection of)
the blow up of $V(n,d)$ along $I\cap V(n,d)$.  Since $I$ is artinian, in the toric case $\varphi_{I_d^{-1}}$ is never regular, because $I$ has to contain the $d$-th powers of the variables. On the contrary, the map $\varphi_{I_d}$ is regular.

In this situation 
we assume that   all $2$-osculating spaces of $X$ have dimension strictly
less that ${n+2}\choose 2$; i.e. $X$ satisfies a Laplace equation of order $2$. Since
the $2$-osculating spaces of $V(n,d)$ have the expected dimension, this means that $I$
meets the $2$-osculating space $\TC_x^{(2)}V(n,d)$ for all $x\in V(n,d)$.

Let $d=3$. $V(n,3)\subset \PP(Sym^3(V^*))$ represents the homogeneous polynomials of degree $3$ which are cubes of a linear form. Let $\sigma_2V(n,3)$ denote its secant variety; its general element can be interpreted both as a sum of two cubes of linear forms and as a product of three linearly dependent linear forms. Let $\pi_{I_3}: V(n,3)\dashrightarrow X$ denote the projection with center $I_3$. We connect the singularities of $X$ to the reciprocal position of $I_3$ and  $\sigma_2V(n,3)$.
\begin{prop}\label{secant}
If $I\cap\sigma_2V(n,3)$ strictly contains $\sigma_2(I\cap V(n,3))$ then $X$ is singular.
\end{prop}
\begin{proof}
The points of $I\cap\sigma_2V(n,3)$ give rise to nodes of $X$, except those of $\sigma_2(I\cap V(n,3))$, because $I\cap V(n,3)$ is the indeterminacy locus of $\pi_{I_3}$. Note that $\sigma_2(I\cap V(n,3))\subset I$, because $I$ is an ideal.
\end{proof}

Among Togliatti systems, 
not necessarily monomial, we detect two kinds which we call trivial ones.

\begin{defn} A Togliatti system of forms of degree $d$ is trivial of type A if
there exists a form $Q$ of degree $d-1$  such that, for every $L\in V^*$, $QL\in I$,
that is $Q$ belongs to the saturation of $I$.
\end{defn}

Note that the ideal generated by a quadratic form $Q$ defines a trivial Togliatti
system of cubics of type A which is not artinian, but adding $s\ge n$ suitable forms
$F_1,\ldots ,F_s\in Sym^3V^*$ we get a linear system $Q\langle x_0,\ldots ,x_n\rangle+\langle F_1,\ldots
,F_s\rangle$ which is an artinian trivial Togliatti system of type A.

In the toric case, if $Q$ is a quadratic monomial, then $Q$ has rank $\le 2$,
therefore $I=(Q)+(F_1,\ldots ,F_s)$ meets $\sigma_2V(n,3)$ in infinitely many points
outside $I$.
In particular, by Proposition  \ref{secant}, a toric trivial Togliatti system of cubics of type
A cannot parameterize a smooth variety.

\begin{ex}\rm
Consider the $12$-dimensional
linear system of cubics
$$\cL =\langle x_0^2x_1, x_0^2x_2,x_0^2x_3,x_1^2x_0,
x_1^2x_2,x_1^2x_3,x_2^2x_0, x_2^2x_1,x_2^2x_3,x_0x_1x_3,
x_0x_2x_3, x_1x_2x_3 \rangle \subset \cL_{3,3}.$$ Let $\varphi
_{\cL}:\PP^3 \longrightarrow \PP^{11}$ be the rational map
associated to $\cL$. Its image $X:=\overline{\Image (\varphi
_{\cL})}\subset \PP^{11}$ is (projectively equivalent to) the
projection  from the linear subspace $$\cL':= \langle x_0^3,
x_1^3, x_2^3, x_3^3, x_0x_1x_2, x_0x_3^2, x_1x_3^2,x_2x_3^2\rangle
$$  of the Veronese variety $V(3,3)\subset \PP( \cL
_{3,3})=\PP^{19}$. We easily check that $X$ is not smooth. In fact
$ Sing(X)=\{(0,0,0,1)\}$. Finally, let us check that $X$ satisfies
a Laplace equation of order 2. Since $x_0^3, x_1^3, x_2^3,
(x_0+x_1+x_2)^3, x_0x_1x_2, x_0(x_0+x_1+x_2)^2,
x_1(x_0+x_1+x_2)^2,x_2(x_0+x_1+x_2)^2$ are $k$-linearly
dependent,  applying Lemma \ref{keylemma} and Proposition
\ref{lem-L-element} we get that the ideal $I=(x_0^3, x_1^3, x_2^3,
x_3^3, x_0x_1x_2, x_0x_3^2, x_1x_3^2,x_2x_3^2 )\subset
R=k[x_o,x_1,x_2,x_3]$ fails the WLP in degree $2$. Therefore,
using that  $\cL $ and $\cL '$ are apolar and Theorem
\ref{teathm}, we conclude that $X$ satisfies a Laplace equation of
order 2. Alternatively, we could observe that $X$ is ruled,
because the variable $x_3$ appears in the polynomials of the
linear system $\cL$  only up to degree $1$, or, alternatively, the
polynomials of $\cL'$ contain all monomials of degree $\geq 2$ in
$x_3$.
\end{ex}

\begin{defn}
A  Togliatti system of forms of degree $d$ is trivial of type B when there is a point $p\in V(n,d)$ such that
the intersection of $I$ with the $(d-1)$-osculating space at  $p$ meets
all the other $(d-1)$-osculating spaces.
\end{defn}

A trivial Togliatti system of type B is given in Example \ref{trivB}.
To explain this, let us recall that, if $p\in V(n,d)$ is identified to $L^d$, where $L\in V^*$, then $\TC_p^{(1)}V(n,d)$ is formed by the multiples of $L^{d-1}$,  $\TC_p^{(2)}V(n,d)$ by the multiples of $L^{d-2}$, and so on. From this description it follows that a sufficient condition to have a Togliatti system of cubics of type B is  $$\dim_k( I\cap \TC_p^{(2)}) > {{n+2}\choose 2}-n-1={{n+1}\choose 2},$$ because this number is the codimension of the intersection of two osculating spaces inside one of them.
 We found several cases when this happens even if
$\dim I\cap \TC_p^{(2)}= {{n+2}\choose 2}-n-1$.

 \begin{rem} \rm G. Ilardi has a different notion of trivial Laplace equations in \cite{I}
 Remark 1.2,
 which corresponds to varieties embedded in a space of dimension smaller than the
 expected dimension
 of the osculating spaces, see Remark \ref{bound}. Still another definition can be
 found in \cite{FI}.
 \end{rem}

\begin{prop}
Let $I$ be a monomial artinian
ideal $I$, generated in degree $3$. Assume that $I$ is trivial of type B of the form
$I=(LF_1, \ldots , LF_t, G_1, \ldots , G_n)$, where $L, F_i, G_j$ are monomials of degrees $1,2,3$ respectively, and $t> {{n+1}\choose 2}$. Then
the variety $X$ is singular.
\end{prop}
\begin{proof}
Since $I$ is monomial, we can assume that $L=x_0$ and $G_i=x_i^3$, for all $i\geq 1$. We want to prove that $I$ meets the tangent space at $p=L^3$ outside $I\cap V(n,3)$, giving a singularity of $X$. We are done if among the polynomials $F_1, \ldots , F_t$ there is a multiple of $x_0$ different from $x_0^2$. In view of the assumption on $t$, the unique case to check separately is when $t={{n+1}\choose 2}+1$ and $\{F_1, \ldots , F_t\} $ contains $x_0^2$ and all monomials of degree $2$ in $x_1, \ldots , x_n$. But in this case, looking at the corresponding polytope $P$, we see that
the vertex $x_0^2x_1$ has edges in $P$ connecting to the $2n-2$ vertices
$x_0^2x_2,\ldots x_0^2x_n, x_1^2x_2,\ldots , x_1^2x_n$, so that  for $n\ge 3$
we get $2n-2>n$, hence
the polytope $P$ is not simple and the variety $X$ is not smooth
(even not quasi smooth) by \cite{GKZ}, chap. 5, Proposition 4.12.
\end{proof}

Note that a monomial  artinian ideal $I$ generated in degree three contains the monomials $x_i^3$ for $i=0,\ldots ,n$.

We are now ready to give  a complete classification of monomial Togliatti systems of cubics in the cases $n=2$ and $3$.

In the case $n=2$, let $k[a,b,c]$ be the base ring, we recall that the only non-trivial monomial Togliatti system  is $a^3, b^3, c^3, abc$
(see \cite{FI,V}). In view of  next classification theorem for $n=3$, we remind
also that all toric surfaces are quasi-smooth according to \cite{GKZ} chap. 5, \S 2.

\begin{thm}\label{toricn3}
Let $I\subset k[a,b,c,d]$ be a monomial artinian
ideal of degree $3$,
 such that the
corresponding threefold $X$
is smooth and does satisfy a Laplace equation of degree $2$.
Then, up to a permutation of the coordinates, $I^{-1}$ is one of the following three examples:

\begin{itemize}
\item[(1)] $(a^2b, a^2c, a^2d, ab^2, ac^2, ad^2, b^2c, b^2d, bc^2, bd^2, c^2d, cd^2)$, $X$ is of degree $23$, in $\PP^{11}$,
it is isomorphic to $\PP^3$ blown up in the $4$ coordinate points;
\item[(2)] $(abc, abd,  a^2c, a^2d,  ac^2, ad^2, b^2c, b^2d, bc^2, bd^2, c^2d, cd^2)$, $X$ is
of degree $18$, in $\PP^{11}$, it is isomorphic to $\PP^3$ blown up in the line $\{c=d=0\}$ and in the two points
$(0,0,1,0)$ and $(0,0,0,1)$;
\item[(3)] $(abc, abd, acd, bcd, a^2c, ac^2, a^2d, ad^2, b^2c, bc^2, b^2d, bd^2)$, $X$ is of degree $13$, in $\PP^{11}$,
it is isomorphic to $\PP^3$ blown up in the two  lines $\{a=b=0\}$ and $\{c=d=0\}$.
\end{itemize}
 \noindent Moreover, if we substitute ``smooth'' with ``quasi-smooth'' (see \cite{GKZ} chap. 5, \S 2)
we have the further cases:
\begin{itemize}
\item[(4)] $(acd, bcd,  a^2c, a^2d,  ac^2, ad^2, b^2c, b^2d, bc^2, bd^2, c^2d, cd^2)$, this example is trivial of type A (indeed the apolar ideal contains $ab*(a,b,c,d)$); $X$ is of degree $18$, in $\PP^{11}$,
 and its  normalization  is isomorphic to $\PP^3$ blown up in the line $\{c=d=0\}$ and in the two points
$(0,0,1,0)$ and $(0,0,0,1)$;
\item[(4')] a projection of case (2) removing one or both of the monomials $abc, abd$, or
a projection of case (3) removing a subset of the monomials $(abc, abd, acd, bcd)$,
or a projection of case (4) removing one or both of the monomials $(acd, bcd)$.
\end{itemize}
\end{thm}

\begin{proof}
Consider the apolar ideal $I$. Since it is monomial and artinian, $I$ contains $(a^3, b^3, c^3, d^3)$ and
$j$ generators more, with $1\le j\le 6$.

Due to the Proposition 1.1 of \cite{P}, in order to check that the four cases satisfy a Laplace equation of degree $2$,
it is enough to check that the vertices points in $\Z^4$ are contained in a quadric.
This is $Q:=2(a^2+b^2+c^2+d^2)-5(ab+ac+ad+bc+bd+cd)$ in the case (1)
(it corresponds to a sphere with the same center of the tetrahedron),
it is $Q+9ab$ (a quadric of rank three) in the case (2),
it is $Q+9ab+9cd=(- 2a - 2b + c + d)(- a - b + 2c + 2d)$  in the case (3),
and it is $ab$ in the case (4). An alternative approach for proving that the four cases satisfy a Laplace equation of
                                degree 2 could be to apply directly Theorem \ref{teathm}(2).

Every case corresponds to a convex polytope contained in the full tetrahedron
with vertices the powers $a^3$, $b^3$, $c^3$, $d^3$.

This tetrahedron has four faces like in the Figure 1.
\begin{figure}
    \centerline{\includegraphics[width=40mm]{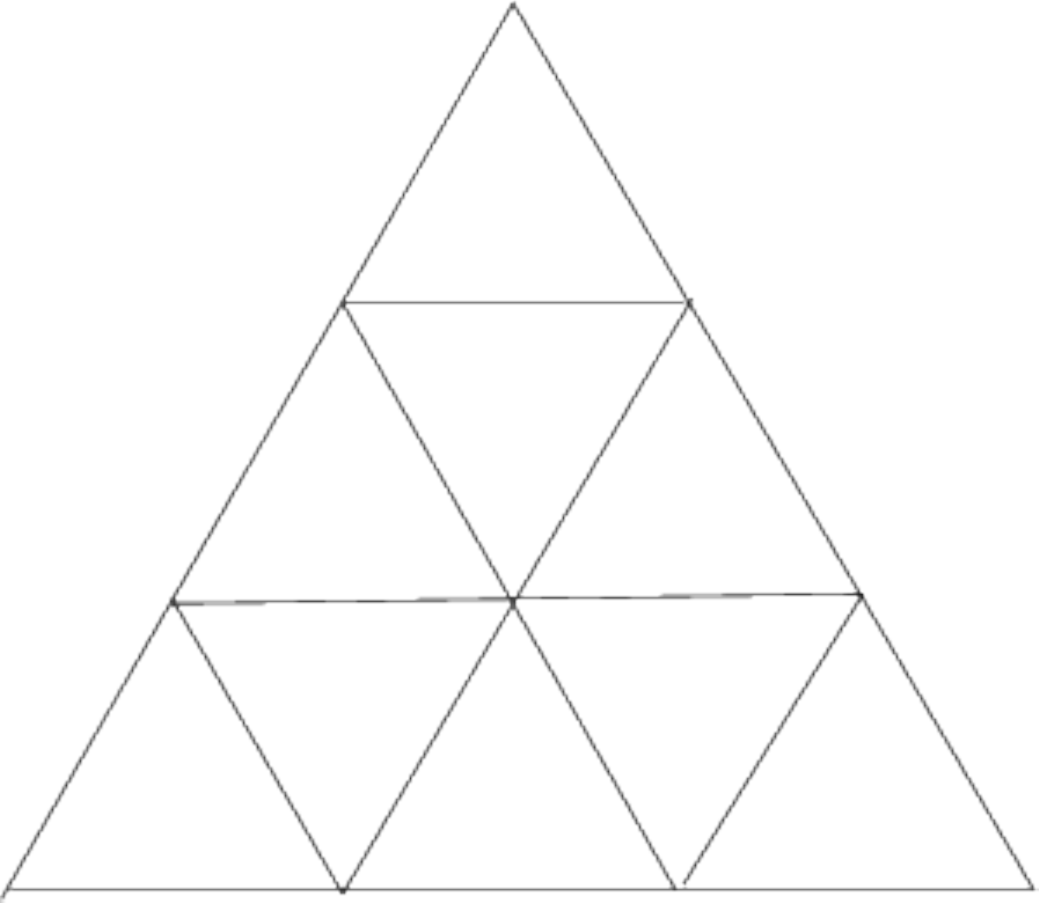}}
    \caption{full triangle}
    \end{figure}

The convex polytope corresponding to the case (1) is the truncated tetrahedron already described.
It is instructive to describe its faces, which are four
``punctured'' hexagons like in the Figure 2
\begin{figure}
    \centerline{\includegraphics[width=40mm]{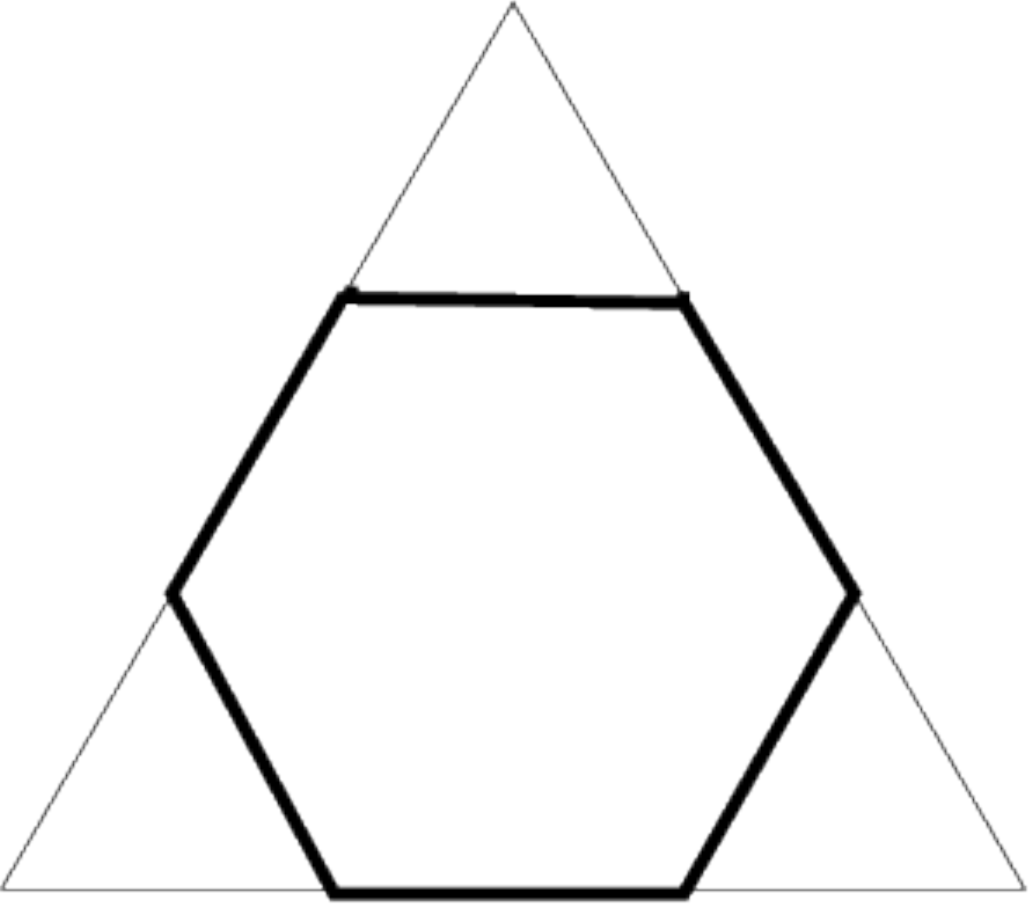}}
    \caption{punctured hexagon}
    \end{figure}
and four smaller regular triangles.
It is the case (4) in the Theorem 3.5 of \cite{P}. It has degree $3^3-4=23$ in $\PP^{11}$.
Note that the projection of this example is not quasi-smooth because, when we remove a vertex,
 the resulting polytope has four faces meeting in a vertex (see \cite{GKZ}, chap. 5, Proposition 4.12).

The case (2) corresponds to the case (5) in the Theorem 3.5 of Perkinson, \cite{P}.
 It has degree $18$ in $\PP^{11}$.
The degree computation follows from the fact that the equivalence of a line
in the (excess) intersection of three cubics in $\PP^3$ counts seven,
according to the Example 9.1.4 (a) of \cite{F}. So $3^3-7-2=18$.

The convex hull has the following faces:
one rectangle, two full trapezoids, two punctured hexagons and
two triangles.

The picture of the full trapezoid is like in the Figure 3
and it is important to remark that all the three vertices of the longer side
are included.
\begin{figure}
    \centerline{\includegraphics[width=40mm]{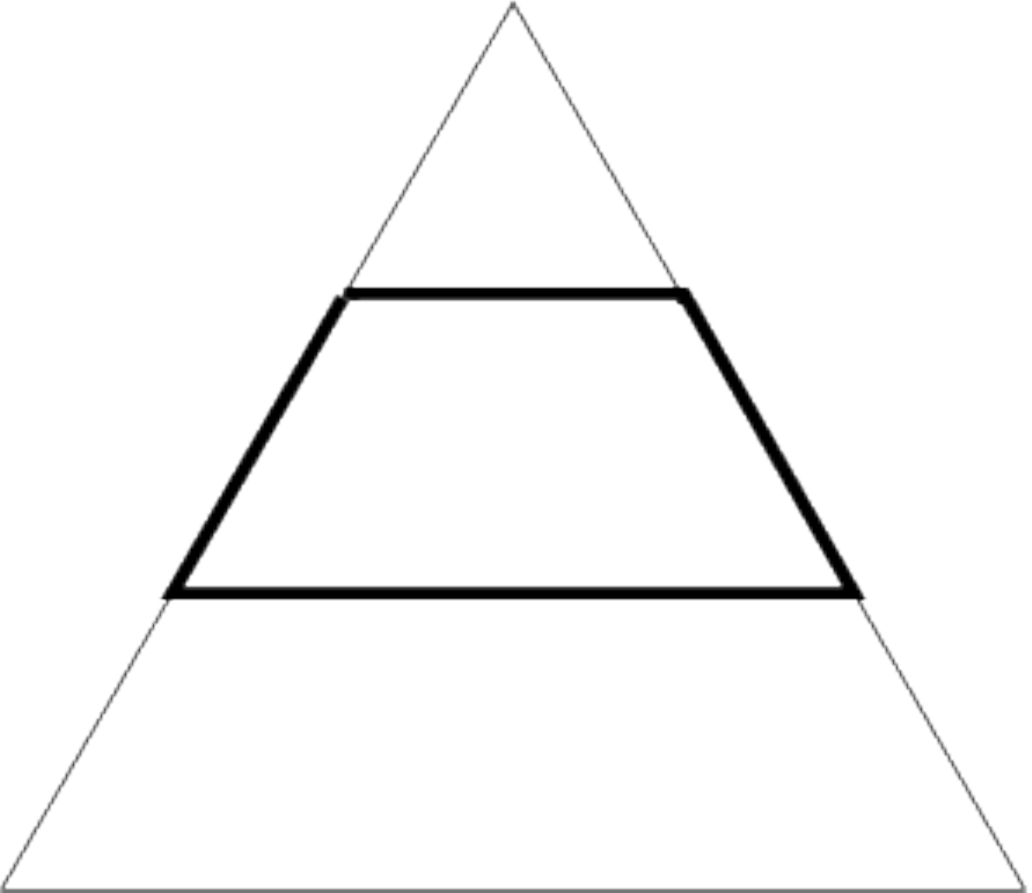}}
    \caption{full trapezoid}
    \end{figure}

The projection of this case is never smooth,
but when removing the mid vertices of the long sides of the trapezoids,
we get a quasi smooth variety, appearing in (4') of our statement.
To understand why these cases are not smooth, note that the condition (a) of the Corollary 3.2 of Chap. 5
of \cite{GKZ} is not satisfied when $\Gamma$ is one of the vertices of the long side of the trapezoid.

The case (3) can be seen both as the case (2) or (3)   in the theorem 3.5
of Perkinson, \cite{P}.
Our variety $X$ is $\PP^3$ blown up on two skew lines $L_1$ and $L_2$.
To see it as a particular case of case (2) of \cite{P}, consider that there are two natural maps from
$X$ to $\PP^1$, with fiber given by
the Hirzebruch surface isomorphic to $\PP^2$ blown up in one point.

Fix a line $L_i$. The map takes a point $p$ to the plane spanned by $L_i$ and p.
These planes through $L_i$ make the target $\PP^1$.

To see it as a particular case of case (3) of \cite{P}, consider that through a general point $p$ there is a unique line meeting
$L_1$ in $p_1$ and $L_2$ in $p_2$.
The map from $X$ to the quadric surface $\PP^1\times\PP^1$
takes  $p$ to the pair $(p_1, p_2)$.

The convex polyhedron has six faces, four full trapezoids and two full (long) rectangles.
The argument regarding the projection is analogous to the previous case and we omit it.

The case (4) does not appear  in the Theorem 3.5 of Perkinson, \cite{P} because it is not smooth.

The convex hull has the following faces:
one rectangle, two punctured trapezoids, two full hexagons and
two triangles.
The presence of the punctured trapezoids is crucial for the non smoothness, exactly as
we saw in the projection of the case (2).

A computer check shows that this list is complete,
in all the remaining cases the convex polytope has at least four faces meeting in some vertex.

Let us just underline
that there are
exactly four monomial Togliatti (cubic) systems with $13$ generators,
their apolar ideals are obtained by adding to $(a^3, b^3, c^3, d^3)$ the monomials $a^2*(b, c, d)$ and their cyclic permutations.
They are trivial of type $A$.

The faces are three full trapezoids, one full hexagon.
The convex hull is topologically equivalent to the Figure 4,
where the four meeting faces are evident,
so it is not quasi smooth.
\begin{figure}\label{hexatri}
    \centerline{\includegraphics[width=40mm]{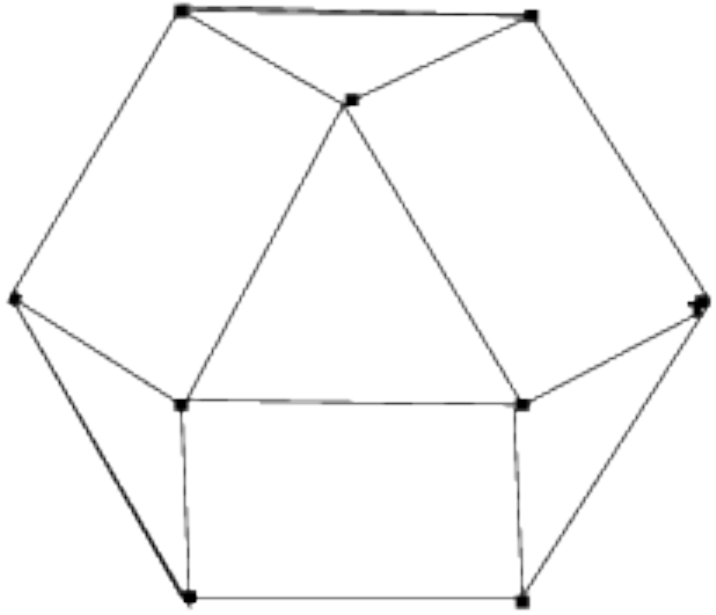}}
\caption{non quasi smooth cases with $13$ vertices}
        \end{figure}
\end{proof}

\begin{rem} \rm
The computations have been performed using Macaulay2 \cite{M}.
\end{rem}


\section{Bounds on the number of generators} \label{bounds}

In this section we concentrate on the case $n=2$. We will see how, using Theorem
\ref{teathm}, it is possible to translate in geometric terms a result expressed in
purely algebraic terms involving WLP.

Let $\cL$  be a linear system  of curves of degree $d$ and (projective) dimension
$N\leq {{d+1}\choose{2}}-1$, defining a map  $\phi_{\cL}: \PP^2 \rightarrow \PP^N$
having as image a surface $X$ which satisfies exactly one Laplace equation of order
$d-1$.

With notations as in Theorem \ref{teathm}, let $I^{-1}$ be the ideal generated by the
equations of the curves in $\cL$ and $I$ its apolar system, generated by $r$
polynomials.

Note that if $\cL$ is  a Togliatti system with $r=3$ ,  then $\cL$ is trivial of type
A and $I$ is not artinian. The Togliatti example described in Remark \ref{Tog} is a
non trivial example with $r=4$ and $I$ artinian. It is a classical result that this is
the only non trivial example with $d=3$ (see \cite{T2} and \cite{FI}).

We consider now the case $r=4$ with $d\geq 4$.

\begin{thm}\label{r=4}
Let $I\subset R:=k[x,y,z]$ be an artinian ideal generated by
$4$ homogeneous polynomials of degree $d\geq 4$.
 Then
\begin{itemize}
\item[(1)] $I$ satisfies the WLP in degree $d-1$,
\item[(2)] if $d$ is not multiple of $3$, then $I$  satisfies the WLP
everywhere,
\item[(3)] if $d$ is multiple of $3$ but not of $6$, then there exists
$I$ which fails the WLP.
\end{itemize}
\end{thm}
\begin{proof}
Let $I=(F_1,\ldots,F_4)$ and denote by $\cE$ the syzygy bundle of $F_1,\ldots,F_4$;
i.e. $\cE$ is  the rank three
bundle on $\PP^2$ with $c_1(\cE)=-4d$, which enters in the exact sequence
\begin{equation} 0\rightarrow \cE\rightarrow \cO_{\PP^2}(-d)^4 \rightarrow
\cO_{\PP^2}\rightarrow 0.
\end{equation}
From \cite{BK}, Theorem 3.3, if $\cE$ is not semistable, then $I$ has the WLP.
So we assume that $\cE$ is semistable and consider the normalized bundle
$\cE_{norm}=\cE(k)$ with $k=[4d/3]$. We distinguish three cases, according to the
congruence class of $d$ modulo 3.
If $d\equiv 1$ mod $3$, then $c_1(\cE_{norm})=-1$, hence by the Theorem of
Grauert-M\"ulich it follows that
the restriction of $\cE_{norm}$ to a general line $L$ is $\cE_{norm}\mid_L\simeq
\cO_L^2\oplus \cO_L(-1)$. Then by \cite{BK}, Theorem 2.2,  $I$ has WLP. Similarly, if
$d\equiv 2$ mod $3$, then $c_1(\cE_{norm})=-2$, and on a general line
$\cE_{norm}\mid_L\simeq \cO_L\oplus \cO_L(-1)^2$. Finally, assume that $d=3\lambda$,
$\lambda\geq 2$. There are two possibilities for $\cE_{norm}\mid_L$: it is isomorphic
either to $\cO_L^3$, and we conclude as in the two previous cases, or to
$\cO_L(-1)\oplus\cO_L\oplus \cO_L(1)$. Hence $\cE\simeq\cE_{norm}(-4\lambda)$ and
$\cE\mid_L\simeq
\cO_L(-1-4\lambda)\oplus\cO_L(-4\lambda)\oplus\cO_L(1-4\lambda)$.
Consider the  exact sequence
\begin{equation}\label{ex}
0\rightarrow \cE\rightarrow \cE(1)\rightarrow \cE\mid_L(1)\rightarrow 0,
\end{equation} and its twists.
The only critical situation is obtained twisting by $4\lambda-2$, it is isomorphic to
\begin{equation}
0\rightarrow \cE(4\lambda-2)\rightarrow \cE(4\lambda-1)\rightarrow
\cO_L(-2)\oplus\cO_L(-1)\oplus\cO_L\rightarrow 0,
\end{equation}
where the second arrow is the multiplication by $L$. By the semistability of $\cE$ we
get $H^0(\cE(4\lambda-2))=H^0(\cE(4\lambda-1))=(0)$. Also $H^2(\cE(4\lambda-2))=(0)$:
indeed, by Serre's duality,
$H^2(\cE(4\lambda-2))\simeq H^0(\cE^*(-4\lambda-1))$, and this is zero by the
semistability of $\cE^*$, because
$c_1(\cE^*(-4\lambda-1)=-3$. Therefore the cohomology exact sequence of (\ref{ex})
becomes:
 \begin{equation}
 0\rightarrow k\rightarrow H^1(\cE(4\lambda-2))\rightarrow
 H^1(\cE(4\lambda-1))\rightarrow k \rightarrow 0.
 \end{equation}
 where $k$ is the base field.
 But $H^1(\cE(4\lambda-2))\simeq (R/I)_{4\lambda-2}$ and $H^1(\cE(4\lambda-1))\simeq
 (R/I)_{4\lambda-1}$, so $I$ fails WLP in degree $4\lambda-2=d+(\lambda-2)$. With
 similar arguments we get that this is the only degree in which $I$ fails WLP, so in
 particular WLP always holds in degree $d-1$.
 Finally, Corollary 7.4 of \cite{MMN2} shows that the ideal $(x^d,y^d,z^d,x^\lambda
 y^\lambda z^\lambda)$, with $d=3\lambda$ odd, fails WLP.
 \end{proof}

 \begin{rem}\label{optimal} \rm
 (1)   Part (2) of Theorem \ref{r=4} was stated for the monomial case
  in \cite{MMN2}; Theorem 6.1. Analogous proof holds for homogeneous polynomials
  non necessarily monomials and we include here for seek of completeness.

(2) U. Nagel has pointed out to us that if $d$ is a multiple of 6 and $I$ is a monomial ideal 
then $I$ does have the WLP. This follows from
Theorem 6.3 in \cite{CN}.

  (3) Theorem \ref{r=4} is optimal, i.e. for all $d\geq 4$ and
$5\leq r\leq d+1$ there exist examples of ideals
$I$ generated by $r$ polynomials
of degree $d$ which fail the WLP in degree $d-1$.

Let $I=(xF,yF,zF,G_1,\ldots ,G_{r-3})$ where $F$ is a homogeneous polynomial with
$\deg F=d-1$ and
$G_1,\ldots ,G_{r-3}$ are general forms of degree $d$. I is an artinian ideal because
$r\geq 5$, and $I^{-1}$ defines
a surface satisfying a Laplace equation of order $d-1$.
\end{rem}

Hence, applying Theorems \ref{teathm}  and \ref{r=4} we get that there do not exist
surfaces $X\subset \PP^{{{d+2}\choose{2}}-4}$ with all $(d-1)$-th osculating spaces
 of dimension less than expected, while there exist examples of such surfaces
in $\PP^N$ for all $N<{{d+2}\choose{2}}-4$.

  We observe that it possible to find smooth surfaces as in Remark \ref{optimal}, for
  instance taking $F=x^{d-1}+y^{d-1}+z^{d-1}$
   and $G_1=x^d$, $G_2=y^d$, $G_3=z^d$.


   \section{Final Comments} \label{final comments}

   A further interesting project is the classification of all Togliatti linear systems of cubics on $\PP^n$, in the monomial case, accomplished here for $n\le 3$ (see Theorem \ref{toricn3}).   It is possible to generalize the three examples in Theorem \ref{toricn3} constructing
   suitable projections of blow ups of $\PP^n$ along unions of linear spaces of
   codimension $\geq 2$ corresponding to partitions of the $n+1$ fundamental points.

   Among the three examples in Theorem \ref{toricn3}, the third one is a ruled threefold,
   while the first two are not. How can we distinguish the ruled examples from the
   non-ruled ones? Since all ruled varieties satisfy Laplace equations of all orders (see
   Remark \ref{ruled}), the non-ruled ones are much more interesting to find.

   The    second
case of Theorem   \ref{toricn3}    generalizes to $n\ge 4$ and
   gives for any $n\ge 3$ a counterexample to Ilardi's conjecture in \cite{I}; pag. 12.
In fact, we have

\begin{ex}  \rm   \label{inpn}
We consider the monomial artinian ideal
$$I=(x_0,x_1,...,x_{n-2})^3+
(x_{n-1}^3,x_n^3,x_0x_{n-1}x_n,x_1x_{n-1}x_n,...,x_{n-2}x_{n-1}x_n)\subset
k[x_0,...,x_n].$$  Since $\dim I_3={n+1\choose 3}+n+1$, we get
that $\dim (I^{-1}_3)=n(n+1)$. Let $X$ be the closure of the image
of $\varphi _{I^{-1}_3} $ which can be seen as the projection of
$V(n,3)$ from $I_3$. $X$ is a smooth $n$-fold in $\PP^{n(n+1)-1}$
isomorphic to $\PP^n$ blown up at the linear space 
$x_{n-1}=x_n=0$ and in the two points $(0,...,0,1,0)$ and
$(0,...,0,1)$. Moreover, it easily follows from Theorem
\ref{teathm} that $X$ satisfies a Laplace equation of degree 2.
A quadric in $\mathbb Z^{n+1}$ containing the
vertices of the corresponding polytope, analogous to the one in the proof of Theorem 4.10, case
(2), has equation: $2(x_0^2+\ldots
+x_n^2)-5(\sum_{i,j=0, i<j}^ n
x_ix_j)+9(\sum_{i,j=0, i< j}^{n-2} x_ix_j)=0$.\end{ex}

\begin{rem}\label{finalrem}\rm
The examples of Theorem \ref{toricn3}  and Example \ref{inpn} can be seen as special cases of a  class of smooth monomial Togliatti systems of cubics. Let $E_0, \ldots, E_n$ be the fundamental points in $\PP^n$, and let $\Pi$ be a partition of the set $\{E_0, \ldots,E_n\}$ such that each part contains at
most $n-1$ points. Let us consider the blow up of $\PP^n$ along  the linear subspaces generated by the parts of $\Pi$ and its embedding with the cubics. Since we are performing a blow up along a torus invariant subscheme, we get a toric variety, which corresponds to a polytope $P$: it is the $n$-dimensional simplex truncated along the faces associated to the blown up spaces. Finally let us consider the projection from the points corresponding to the centres of the full hexagons in $P$. The toric variety $X$ obtained in this way is smooth.
We conjecture that all smooth monomial Togliatti systems of cubics are obtained in this way.
\end{rem}


\end{document}